\documentclass[reqno,12pt,twoside]{article}
\usepackage{fullpage}
\usepackage{amssymb}
\usepackage{amsmath}
\usepackage{amsthm}
\numberwithin{equation}{section}
\newtheorem{theorem}{Theorem}

\newtheorem{lemma}{Lemma}
\newcommand{\Nzero}{\mathbb N_0}

\newcommand{\card}[1]{\left|#1\right|}
\newcommand{\fr}[1]{\left\{#1\right\}}

\begin{document}

\title{\Large On a problem of Erd\H{o}s and Nathanson related to minimal asymptotic bases of order $h$
\footnote{The first author is supported by the National Natural Science Foundation of China (Grant No.
12301003), the Anhui Provincial Natural Science Foundation (Grant No. 2308085QA02) and
the University Natural Science Research Project of Anhui Province (Grant No. 2022AH050171). The second author is supported by the National Natural Science Foundation of China (Grant No.
12371005).}}

\author{\large Shi-Qiang Chen$^{a}$,
\large ~~~Quan-Hui Yang$^{b,}$\thanks{Corresponding author. E-mail addresses: csq20180327@163.com (S.-Q. Chen), yangquanhui01@163.com (Q.-H. Yang)}}
\date{} \maketitle
 \vskip -3cm
\begin{center}
\vskip -1cm { \small
\begin{center}
 $a$. School of Mathematics and Statistics,
\end{center}
\begin{center}
Anhui Normal University, Wuhu 241002, P.R. China
\end{center}}
\vskip -1cm { \small
\begin{center}
 $b$. School of
Mathematical Sciences and Ministry of Education Key Laboratory for NSLSCS,
\end{center}
\begin{center}
Nanjing Normal University, Nanjing 210023,
P.R. China
\end{center}}
\end{center}

{\bf Abstract.} Let $h\geq 3$ be an integer and $0<\alpha<1/h$. In this paper,  we prove that there exists a
minimal asymptotic basis of order $h$ with asymptotic density $\alpha$. This solves an open problem posed by Erd\H{o}s and Nathanson in 1988.



\section{Introduction}
Let $\mathbb N$ denote the set of positive integers and put
$\Nzero=\mathbb N\cup\{0\}$. For $A\subseteq\Nzero$, a real number
$x\geq 1$, and a positive integer $h$, define
\[
A(x)=\card{A\cap[1,x]}
\]
and
\[
hA=\{a_1+\cdots+a_h:a_1,\ldots,a_h\in A\}.
\]
We call $A$ an asymptotic basis of order $h$ if $hA$ contains every
sufficiently large integer. An asymptotic basis $A$ of order $h$ is called
minimal if no proper subset of $A$ is an asymptotic basis of order $h$.
The study of minimal asymptotic bases is one of the classical topics in
additive number theory, and a basic problem is to understand how dense a
minimal asymptotic basis can be. 

Let $A$ be a set of integers. The \emph{lower asymptotic density} of $A$ is defined by
\[
d_L(A)=\liminf_{x\to\infty}\frac{A(x)}{x}.
\]
If the limit
\[
\lim_{x\to\infty}\frac{A(x)}{x}
\]
exists, it is called the \emph{asymptotic density} of $A$ and is denoted by $d(A)$. 

For real numbers $a<b$, $(a,b)$ denotes the open interval with endpoints $a$ and $b$, $(a,b]$ denotes the left-open, right-closed interval and $[a,b)$ denotes the left-closed, right-open interval with endpoints $a$ and $b$. In 1989, Nathanson and S\'ark\"ozy~\cite{NathansonSarkozy1989} proved that every
minimal asymptotic basis $A$ of order $h$ satisfies $d_L(A)\leq 1/h$.
In 1988, Erd\H{o}s and Nathanson~\cite{EN} constructed minimal asymptotic bases
$A$ of order $h$ with $d(A)=1/h$. They also proved that for every
$\alpha\in\left(0,1/(2h-2)\right)$, there exists a minimal asymptotic basis $A$
of order $h$ with $d(A)=\alpha$. In particular, for $h=2$ this gives every
$\alpha\in(0,1/2]$. Naturally, they posed the following well-known open problem.\\
{\bf Erd\H{o}s--Nathanson's Problem.}
Let $h\geq 3$ be an integer. If $\alpha\in(0,1/h)$, prove that there exists a
minimal asymptotic basis $A$ of order $h$ with asymptotic density $\alpha$.

For other related results on minimal asymptotic bases, see [1-3, 5-10, 12].

In this paper, we solve this problem. 
\begin{theorem}\label{thm1}
Let $h\geq3$ be an integer and let $0<\alpha<1/h$. Then there exists a minimal asymptotic
basis $A'$ of order $h$ such that $d(A')=\alpha$.
\end{theorem}

For an integer $k\geq 2$, a sequence $B=\{b_i\}_{i\geq1}$ of positive
integers is called a $B_k$-sequence if
every sum of $k$
elements of $B$, written in nondecreasing order, has a unique
representation. We use the elementary
lacunarity fact that if $b_{i+1}>k b_i$ for every $i\geq1$, then
$B$ is a $B_k$-sequence. Moreover, such a sequence satisfies
$B(x)=O(\log x)$.

\section{Proof of Theorem~\ref{thm1}}

\begin{lemma}(See \cite[Theorem~1]{EN})\label{lem1}
Let $h\geq2$, and let $A=B\cup C$ be an asymptotic basis of order $h$, where
$B$ and $C$ are disjoint. Let $r(n)$ be the largest number of pairwise
disjoint representations
\[
 n=b_1+\cdots+b_{h-1}+c,
 \qquad b_i\in B,\quad b_1<\cdots<b_{h-1},\quad c\in C.
\]
Let $W$ be the set of those $w\in hA$ for which every representation of $w$
by $h$ elements of $A$ contains at most one element of $C$, and put
\[
 \Omega(n)=\{c\in C:n-c\in(h-1)B\}.
\]
Suppose that, for some $\delta\in(0,1)$, the following properties hold.
\begin{enumerate}
 \item If $B=\{b_i\}_{i\geq1}$ in increasing order, then
       $b_{i+1}>(2h-2)b_i$.
 \item We have $r(n)\to\infty$ as $n\to\infty$.
 \item For every $c\in C$, there are infinitely many choices of
       $b_1,\ldots,b_{h-1}\in B$ for which
       \[
       w=b_1+\cdots+b_{h-1}+c\in W\setminus hB
       \]
       and $c'>\delta w$ for every $c'\in\Omega(w)\setminus\{c\}$.
 \item For every $b_1\in B$, at least one of the following alternatives
       occurs infinitely often:
       \begin{enumerate}
       \item there are $b_2,\ldots,b_{h-1}\in B$ and $c\in C$ such that
             \[
             w=b_1+\cdots+b_{h-1}+c\in W\setminus hB
             \]
             and $c'>\delta w$ for every
             $c'\in\Omega(w)\setminus\{c\}$;
       \item there are $b_2,\ldots,b_h\in B$ such that
             \[
             w=b_1+\cdots+b_h\in W
             \]
             and $c'>\delta w$ for every $c'\in\Omega(w)$.
       \end{enumerate}
\end{enumerate}
Then there is a set $C'\subseteq C$ such that $B\cup C'$ is a minimal
asymptotic basis of order $h$ and
\[
 (C\setminus C')(x)\leq 2B(x)^{h-1}
\]
for all sufficiently large $x$. In particular, if $B(x)=O(\log x)$, then
$d(C\setminus C')=0$.
\end{lemma}

\begin{lemma}(See  \cite{Weyl1916})\label{lem2}
Let $\theta$ be an irrational number. Then the sequence
$\{n\theta\}$, $n=1,2,\ldots$, is uniformly distributed modulo $1$.
More precisely, for every interval $[a,b)\subseteq[0,1)$, we have
\[
 \bigl|\{1\leq n\leq N:a\leq \{n\theta\}<b\}\bigr|
 =(b-a)N+o(N)
\]
as $N\to\infty$. In particular, the sequence $\{n\theta\}$ visits every
nonempty open interval in $(0,1)$ infinitely often.
\end{lemma}

\begin{proof}[Proof of Theorem~\ref{thm1}]
Put $\rho=h\alpha$. Noting that $0<\alpha<1/h$, we have $\rho\in(0,1)$.
Let
\[
 \mathcal I=\{(a,b):a,b\in\mathbb Q\cap[0,1],\ a<b\}.
\]
Choose a sequence $(I_i,\varepsilon_i)_{i\geq1}$ in
$\mathcal I\times\{0,1\}$ such that every pair $(I,\varepsilon)$ occurs
infinitely often. Fix an irrational number $\theta$. By Lemma~\ref{lem2}, the sequence
$\{q\theta\}$, $q\in\mathbb N$, visits every nonempty open interval
infinitely often. Hence, choose $q_1\in\mathbb N$ with
$\fr{q_1\theta}\in I_1$, and let
\[
 b_1=hq_1+\varepsilon_1.
\]
We define the remaining numbers recursively. Assume that $b_i$ and
$q_i\in\mathbb N$ have been chosen with
\[
 b_i=hq_i+\varepsilon_i.
\]
Choose $q_{i+1}\in\mathbb N$ such that
\[
 \fr{q_{i+1}\theta}\in I_{i+1}
 \qquad\text{and}\qquad
 hq_{i+1}+\varepsilon_{i+1}>(2h-1)b_i,
\]
and let $b_{i+1}=hq_{i+1}+\varepsilon_{i+1}$.
 We define

\[
 B=\{b_i\}_{i\geq1}.
\]
Then
\begin{equation}\label{eq1}
 b_{i+1}>(2h-1)b_i
\end{equation}
for all integers $i\geq1$.
Moreover, for each $\varepsilon\in\{0,1\}$ and each nonempty open interval $I'\subset (0,1)$, we have
\begin{equation}\label{eq2}
| \bigl\{\fr{q_i\theta}:\varepsilon_i=\varepsilon\bigr\}\cap I'|=\infty.
\end{equation}
Let
\[
 B_\varepsilon=\{b_i\in B:b_i\equiv\varepsilon\pmod h\}
 \quad(\varepsilon=0,1),
\]
and define
\[
 C_0=\{hm:m\in\Nzero,\ \fr{m\theta}<\rho\},
 \qquad C=C_0\setminus B_0,
 \qquad A=B\cup C.
\]
Then $B\cap C=\emptyset$. By Lemma~\ref{lem2}, we have
\[
 C_0(x)=\rho\frac{x}{h}+o(x)=\alpha x+o(x).
\]
On the other hand, by \eqref{eq1},  we have $B(x)=O(\log x)$. Therefore,
\begin{equation}\label{eq3}
 d(C)=d(A)=\alpha.
\end{equation}

We verify that the hypotheses of Lemma \ref{lem1} hold. By \eqref{eq1}, condition (1) is true. 

Since
\[ b_{i+1}>(2h-1)b_i>(2h-2)b_i,~~~b_{i+1}>(2h-1)b_i>hb_i>(h-1)b_i,
\]
it follows that $B$ is a $B_{2h-2}$-sequence and hence both a $B_{h-1}$-sequence and a
$B_h$-sequence. Fix an integer $M\geq1$ and choose an integer $K$ such that
\begin{equation}\label{eq4}
 \frac{2}{K}<\rho.
\end{equation}
For every $e\in\{0,1,\ldots,h-1\}$, $k\in\{1,\ldots,K\}$, and
$j\in\{1,\ldots,M+1\}$, choose a block
\[
 T_{e,k,j}=\{b_{1,e,k,j},\ldots,b_{h-1,e,k,j}\}\subseteq B
\]
so that all the blocks are mutually disjoint, exactly $e$ members of the
block belong to $B_1$, and the other $h-1-e$ members belong to $B_0$.
Writing
\[
 b_{i,e,k,j}=hq_{i,e,k,j}+\varepsilon_{i,e,k,j},
\]
we additionally require
\begin{equation}\label{eq5}
 \fr{q_{i,e,k,j}\theta}\in
 \left(\frac{k-1}{(h-1)K},\frac{k}{(h-1)K}\right)
 \qquad(1\leq i\leq h-1).
\end{equation}
Property \eqref{eq2} makes all these finitely many, mutually disjoint
choices possible.

Let
\[
 s_{e,k,j}=\sum_{i=1}^{h-1}b_{i,e,k,j}
           =hQ_{e,k,j}+e.
\]
Since the sum of the $h-1$ intervals in \eqref{eq5} is contained
in 
\[
((k-1)/K,k/K)\subset(0,1),
\]
it follows that
\begin{equation}\label{eq6}
 \fr{Q_{e,k,j}\theta}\in
 \left(\frac{k-1}{K},\frac{k}{K}\right).
\end{equation}
Noting that $B$ is a
$B_{h-1}$-sequence, the numbers $s_{e,k,j}$ are distinct for distinct $j$.  Let 
\[S=\max\{ s_{e,k,j}:e\in\{0,1,\ldots,h-1\}, k\in\{1,\ldots,K\}, j\in\{1,\ldots,M+1\}\}.
\]

Take $n>2S$ and write $n=hL+e$, where $0\leq e<h$. Put
$x=\fr{L\theta}$. If $x\in[1/K,1)$, choose
$k\in\{1,\ldots,K-1\}$ with
\[
 \frac{k}{K}\leq x<\frac{k+1}{K},
\]
it follows from
\eqref{eq6} that
\[
 0<\fr{(L-Q_{e,k,j})\theta}<\frac{2}{K}.
\]
If $x\in[0,1/K)$, choose $k=K$, then
\[
 0<\fr{(L-Q_{e,k,j})\theta}<\frac{2}{K}.
\]
Therefore, by \eqref{eq4}, we have
\[
 c_j=n-s_{e,k,j}=h(L-Q_{e,k,j})>S>0
\]
and
\[
 0<\fr{(L-Q_{e,k,j})\theta}<\frac{2}{K}<\rho
\]
for every $j\in\{1,\ldots,M+1\}$. Hence
$c_j\in C_0$ for every $j\in\{1,\ldots,M+1\}$. Moreover, no $c_j$ is one of the previously chosen block elements.

Since $B$ is a $B_{h}$-sequence, it follows that
\[|\{c_1,\ldots,c_{M+1}\}\cap B_0|\leq1.
\]
Hence
\[|\{c_1,\ldots,c_{M+1}\}\cap (C_0\setminus B_0)|\geq M.
\]
The corresponding representations are pairwise disjoint, since the blocks
are mutually disjoint, the numbers $c_j$ are distinct, and every $c_j>S$
exceeds every block element. Thus, $r(n)\geq M$ for every sufficiently large $n$. Noting that $M$ is arbitrary, we have
\begin{equation}\label{eq7} 
 r(n)\longrightarrow\infty.
\end{equation}
In particular, $A$ is an asymptotic basis of order $h$. Thus, condition (2) is true.

Since $c\equiv0\pmod{h}$ for every $c\in C$ and $a\equiv0,1\pmod{h}$ for every $a\in A$, it follows that:
\begin{equation}\label{eq8}
 w\in hA,\quad w\equiv h-1\pmod h
 \qquad\Longrightarrow\qquad w\in W.
\end{equation}
Let
\begin{equation}\label{eq9}
 \delta=\frac{1}{h+1}.
\end{equation}
For $t\geq2$, by \eqref{eq1} and $h>2$, we have
\begin{equation}\label{eq10}
 b_t-b_{t-1}>
 \frac{2h-2}{2h-1}b_t>
 \frac{h}{h+1}b_t.
\end{equation}
Fix $c\in C$. Choose arbitrarily large index $t\geq2$ such that $b_t\in B_1$ and $b_t>c$, and put
\begin{equation}\label{eq11}
 w=(h-1)b_t+c.
\end{equation}
Then $w\equiv h-1\pmod h$, and it follows from \eqref{eq8} that $w\in W$. Also,
\[
 (h-1)b_t\leq w<hb_t<b_{t+1}.
\]
We claim that $w\notin hB$.  Assume that $w\in hB$. Write
\[w=b_1+b_2+\ldots+b_h,~~~b_i\in B,~~i=1,2,\ldots,h.
\]
Then $b_i\leq b_t$ for $i=1,2,\ldots,h$.
If $\sum\limits_{i=1}^{h}|\{b_i\}\cap\{b_t\}|=h$, then 
\[w=(h-1)b_t+c=hb_t.\]
So $c=b_t$, a contradiction. If $\sum\limits_{i=1}^{h}|\{b_i\}\cap\{b_t\}|=h-1$, then  $c\in B$, a contradiction. If $\sum\limits_{i=1}^{h}|\{b_i\}\cap\{b_t\}|\leq h-2$, then 
\[
 w\leq (h-2)b_t+2b_{t-1}<(h-1)b_t\leq w,
\]
a contradiction. Hence, $w\in W\setminus hB$.

Let $c'\in\Omega(w)\setminus\{c\}$ and write
\[
 w=b'_1+\cdots+b'_{h-1}+c',\qquad b'_i\in B.
\]
Then $b'_i\leq b_t$ for $i=1,2,\ldots, h-1$. Noting that $c'\neq c$, we have 
\[
\sum\limits_{i=1}^{h-1}|\{b'_i\}\cap\{b_t\}|<h-1.
\] 
It follows from $w<hb_t$ that
\[
 c'\geq w-\bigl((h-2)b_t+b_{t-1}\bigr)
      \geq b_t-b_{t-1}
      >\frac{h}{h+1}b_t
      >\frac{w}{h+1}=\delta w.
\]
Hence, condition
(3) is true.

Fix $b_u\in B_0$, choose arbitrarily large index $t>u$ such that $b_t\in B_1$,
and let
\begin{equation}\label{eq12}
 w=b_u+(h-1)b_t.
\end{equation}
By $w\in hB$ and \eqref{eq8}, we have $w\in W$. If
$c'\in\Omega(w)$, then write
\[ 
w=b'_1+b'_2+\cdots+b'_{h-1}+c'.
\]
By $b_u\notin C$, we have
\[
\sum\limits_{i=1}^{h-1}|\{b'_i\}\cap\{b_t\}|<h-1.
\] 
Exactly as above, by $w<hb_t$, we have 
\[
 c'>b_t-b_{t-1}>\frac{w}{h+1}=\delta w.
\]

Fix $b_u\in B_1$ and choose arbitrarily large index
$t>u$ such that $b_t\in B_1$. By \eqref{eq3}, we have
\[
 \card{C\cap(2b_t-b_u,3b_t-b_u)}
   =\alpha b_t+o(b_t)>0
\]
for all sufficiently large $t$. Choose
\[
 c\in C\cap(2b_t-b_u,3b_t-b_u)
\]
and put
\begin{equation}\label{eq13}
 w=(h-2)b_t+b_u+c.
\end{equation}
Then $w\equiv h-1\pmod h$, and hence $w\in W$. Moreover,
\begin{equation}\label{eq14}
 hb_t<w<(h+1)b_t<(2h-1)b_t<b_{t+1},
\end{equation}
where $h\geq3$ is used in the middle inequality. Thus, $w\notin hB$.

If $c'\in\Omega(w)$, then write
\[ 
w=b'_1+b'_2+\cdots+b'_{h-1}+c'.
\]
By \eqref{eq14}, we have $b'_{i}\leq b_t$ for $i=1,2,\ldots,h-1$.
 Consequently,
\[
 c'\geq w-(h-1)b_t>b_t>\frac{w}{h+1}=\delta w.
\]
Thus, condition (4) is true.

All the hypotheses of Lemma \ref{lem1} are now verified. Therefore
there is a subset $C'\subseteq C$ such that $A'=B\cup C'$ is a minimal asymptotic basis of order $h$, and
\[
 (C\setminus C')(x)\leq2B(x)^{h-1}
 =O\bigl((\log x)^{h-1}\bigr)=o(x).
\]
It follows from \eqref{eq3} that
\[
 d(A')=d(A)=\alpha.
\]

\end{proof}

\end{document}